\documentclass{article}

\usepackage{amsmath,amsfonts,latexsym,amssymb,amsthm,url}
\usepackage{etoolbox}
\usepackage[utf8]{inputenc}

\newcommand{\C}{{\mathbb C}}
\renewcommand{\P}{{\mathbb P}}
\newcommand{\Q}{{\mathbb Q}}
\newcommand{\R}{{\mathbb R}}
\newcommand{\Z}{{\mathbb Z}}

\renewcommand{\Im}{\mathrm{Im}}

\newcommand{\calC}{\mathcal{C}}

\newcommand\eps\varepsilon
\newcommand\ph\varphi

\newcommand{\height}{\mathrm{h}}
\newcommand{\tilheight}{\tilde{\height}}

\newcommand{\bfa}{\mathbf{a}}

\newtheorem{theorem}{Theorem}[section]

\newtheorem{proposition}[theorem]{Proposition}

\newtheorem{corollary}[theorem]{Corollary}

\newtheorem{lemma}[theorem]{Lemma}

\numberwithin{equation}{section}

\makeatletter

\renewcommand*\l@section[2]{%
  \ifnum \c@tocdepth >\z@
    \addpenalty\@secpenalty
    \addvspace{0.2em \@plus\p@}%
    \setlength\@tempdima{1.5em}%
    \begingroup
      \parindent \z@ \rightskip \@pnumwidth
      \parfillskip -\@pnumwidth
      \leavevmode \bfseries
      \advance\leftskip\@tempdima
      \hskip -\leftskip
      #1\nobreak\hfil \nobreak\hb@xt@\@pnumwidth{\hss #2}\par
    \endgroup
  \fi}

\makeatother

\title{Trinomials with given roots}

\author{
Yuri Bilu\footnote{Institut de Mathématiques de Bordeaux, Université de Bordeaux \& CNRS, Talence, France}, 
\stepcounter{footnote}\stepcounter{footnote}
Florian Luca\footnote{School of Mathematics, University of the Witwatersrand, Johannesburg, South Africa;
Research Group in Algebraic Structures and Applications, King Abdulaziz Uniersity, Jeddah, Saudi Arabia;
Department of Mathematics, University of Ostrava, Czech Republic;
supported by IRN ``GANDA'', by CNRS, and by grant no. 17-02804S of the Czech Granting
Agency}
}

\begin{document}

\hfuzz 4pt

\maketitle

\begin{abstract}
We show that, apart from some obvious exceptions, the number of trinomials vanishing at given complex numbers is bounded by an absolute  constant. When the numbers are algebraic, we also bound effectively the degrees and the heights of these trinomials. 

\end{abstract}

{\footnotesize

\tableofcontents

}

\section{Introduction}

Everywhere below a \textit{trinomial} is an abbreviation for \textit{a monic trinomial non-vanishing at~$0$}; in other words,  a polynomial of the form ${X^m+AX^n+B}$ with ${B\ne 0}$. 

Let ${\Omega\subset\C^\times}$ be a finite set of non-zero complex numbers. In this note we study the trinomials vanishing at all elements of~$\Omega$.

Surely, infinitely many trinomials vanish at~$\Omega$ if ${\#\Omega\le 2}$. More generally, call ${\alpha,\beta\in \Omega}$  \textit{equivalent} if ${\alpha/\beta}$ is a root of unity. Then the following statement is obviously  true: if~$\Omega$ splits into~$2$ or less equivalence classes, then infinitely many trinomials vanish at~$\Omega$. Indeed, in this case~$\Omega$ is a subset of the roots of a polynomial of the form ${g(X^k)}$, where~$g$ is of degree~$2$ and~$k$ is some positive integer, and this $g(X^k)$ divides infinitely many trinomials. More generally, if~$\Omega$ splits into ${\ell-1}$ or less equivalence classes, then  infinitely many $\ell$-nomials vanish at~$\Omega$. 

Using known results about linear equations in multiplicative groups, it is not hard to show the following.

\begin{theorem}
\label{thc}
Assume that~$\Omega$ splits into~$3$ or more equivalence classes (in the sense defined above). Then the number of  trinomials vanishing at~$\Omega$ is bounded by an absolute effective constant. 
\end{theorem}

Note that in the statement of this theorem, as well as of the subsequent Theorem~\ref{thalg} we do not formally exclude binomials ${X^m+B}$, which can be viewed as trinomials ${X^m+AX^n+B}$ with  ${A=0}$. However,  a binomial can vanish only at a set~$\Omega$ consisting of a single equivalence class, so the (finitely many) trinomials featured in  these theorems are genuine trinomials, with ${AB\ne 0}$. 

Theorem~\ref{thc} is not really new: it can be obtained by combining the proof of Theorem~1 in~\cite{EGST88} with the results from \cite{AV09,ESS02}, which were not available at the time when~\cite{EGST88} was written. However, since we did not find in the literature exactly this statement, we include a short proof in Section~\ref{sc}. 

Our principal result concerns the case when~$\Omega$ consists of algebraic numbers. We denote by ${\height(\cdot)}$ the absolute logarithmic height of an algebraic number or of a polynomial, see Section~\ref{sheights}. Given a finite set ${\Omega \subset \bar\Q}$, we denote
$$
\height(\Omega)=\max\{\height(\alpha): \alpha\in \Omega\}. 
$$ 

\begin{theorem}
\label{thalg}
In the set-up of Theorem~\ref{thc} assume that the elements of~$\Omega$ generate a number field of degree~$d$. Then every trinomial vanishing at~$\Omega$ is of degree not exceeding $10^{60}e^{10d^2(\height(\Omega)+1)}$ and of height not exceeding $10^{70}e^{10d^2(\height(\Omega)+1)}$. In particular, the problem of determining all trinomials vanishing at~$\Omega$ is decidable. 
\end{theorem}

This theorem is proved in Section~\ref{salg}. 

Let~$K$ be a field of characteristic~$0$ and ${\alpha \in\bar K}$ an element algebraic over~$K$. For certain applications one needs information about trinomials over~$K$ vanishing at~$\alpha$. Clearly, if for some positive integer~$k$ we have ${[K(\alpha^k):K]\le 2}$ then  infinitely many trinomials vanish at~$\alpha$.

\begin{corollary}
Let~$K$ be a field of characteristic~$0$ and let ${\alpha\in \bar K}$ be such that 
\begin{equation}
\label{ehyp}
[K(\alpha^k):K]\ge 3 \qquad (k=1,2,\ldots).
\end{equation}
Then the number of trinomials in $K[X]$ vanishing at~$\alpha$ is bounded by an absolute effective constant.
Moreover, if~$K$ is a number field then each such trinomial is of degree not exceeding $10^{60}e^{10d^2\nu^6(\height(\alpha)+1)}$ and of height not exceeding $10^{70}e^{10d^2\nu^6(\height(\alpha)+1)}$, where ${d=[K:\Q]}$ and ${\nu=[K(\alpha):K]}$. In particular, the problem of determining all trinomials over~$K$ vanishing at~$\alpha$ is decidable. 
\end{corollary}
To deduce the corollary from Theorems~\ref{thc} and~\ref{thalg}, just apply them for the set ${\Omega=\{\alpha,\beta, \gamma\}}$, where~$\beta$ and~$\gamma$ are two conjugates of~$\alpha$ over~$K$ such that none of the quotients ${\alpha/\beta,\alpha/\gamma,\beta/\gamma}$ is a root of unity. Existence of such~$\beta$ and~$\gamma$ follows from hypothesis~\eqref{ehyp}. 

\section{Generalities about heights}
\label{sheights}
In this section we very briefly recall definitions and basic facts about absolute logarithmic heights. The \textit{height of a point ${\bfa=(a_0:\cdots:a_n)\in \P^n(\bar\Q)}$ in the projective space} is defined by 
$$
\height(\bfa) = \frac{1}{[K:\Q]}\sum_{v\in M_K}[K_v:\Q_v]\log\max\{|a_0|_v, \ldots,|a_n|_v\},
$$
where~$K$ is a number field containing ${a_0,\ldots, a_n}$ and the absolute values on~$K$ are normalized to extend standard absolute values on~$\Q$. The right-hand side is independent of the choice of~$K$ and of the homogeneous coordinates ${a_0,\ldots, a_n}$. For ${\alpha \in \bar\Q}$ we abbreviate
${\height(1:\alpha)=\height(\alpha)}$ and call this the \textit{height of the algebraic number~$\alpha$}. 

The \textit{height of a polynomial} with algebraic coefficients is the height of a point in the projective space whose homogeneous coordinates are the coefficients  of the polynomial. 

We use, without special mention, the standard properties of the heights: for ${\alpha, \beta \in \bar\Q}$ we have
$$
\height(\alpha+\beta) \le \height(\alpha)+\height(\beta)+\log 2, \quad \height(\alpha\beta) \le \height(\alpha)+\height(\beta), \quad \height(\alpha^n)=|n|\height(\alpha), 
$$
etc. If~$\alpha$ and~$\beta$ are conjugate over~$\Q$ then ${\height(\alpha)=\height(\beta)}$. This implies, in particular, that ${\height(|\alpha|)\le \height(\alpha)}$ for a complex algebraic number~$\alpha$. 

We also systematically use the \textit{Louiville inequality}: if~$\alpha$ is a non-zero complex algebraic number of degree~$d$ then 
$$
e^{-d\height(\alpha)}\le |\alpha|\le e^{d\height(\alpha)}. 
$$
Moreover, in these estimates one may replace~$d$ by $d/2$  if ${\alpha \notin\R}$ but we will never use it. 

One special case will be frequently used.
\begin{proposition}
\label{prmodnoone}
Let~$\theta$ is a complex algebraic number of degree~$d$ such that ${|\theta|\ne 1}$. Then 
\begin{equation}
\label{emodnoone}
\bigl|1-|\theta|\bigr|\ge e^{-d^2(\height(\theta)+\log2)}. 
\end{equation}
\end{proposition}

\begin{proof}
Note that the degree of ${|\theta|}$ does not exceed~$d^2$. Indeed, this is obvious when ${\theta \in \R}$. And if ${\theta\notin \R}$ then ${|\theta|^2=\theta\bar\theta}$ belongs to the number field ${\Q(\theta, \bar\theta)\cap\R}$  of degree not exceeding ${d(d-1)/2}$. Hence the degree of $|\theta|$ in this case does not exceed ${d(d-1)}$.  

Since  the height of ${1-|\theta|}$ is at most ${\height(\theta)+\log 2}$, the result follows.  
\end{proof}

Philipp Habegger drew our attention to Theorem~2 of Mahler~\cite{Ma64}, which allows one to replace  ${d^2(\height(\theta)+\log2)}$ in~\eqref{emodnoone} by ${O(d^2\height(\theta)+d\log d)}$; for the details see \cite[Lemma~3.2]{Ha18} . This leads to a similar amendment  in the statement of Theorem~\ref{thalg}.

\section{Proof of Theorem~\ref{thc}}
\label{sc}

Our principal tool will be the fundamental result of Evertse, Schlickewei and Schmidt~\cite{ESS02} about linear equations in multiplicative groups. See also Theorem~6.2 of Amoroso and Viada~\cite{AV09} for a quantitative improvement. 

\begin{theorem}
\label{thuneq}
Let~$\Gamma$ be a subgroup of~$\C^\times$ of finite rank~$r$, and  ${a_1, \ldots, a_s\in \C^\times}$. Call a solution ${(x_1, \ldots, x_s)\in \Gamma^s}$ of 
\begin{equation}
\label{euneq}
a_1x_1+\cdots+a_sx_s=0
\end{equation} 
primitive if no proper sub-sum of ${a_1x_1+\cdots+a_sx_s}$ vanishes, and call two solutions ${(x_1, \ldots, x_s),(x_1', \ldots, x_s')\in \Gamma^s}$ proportional if there exists ${\lambda \in \Gamma}$ such that ${x_i'=\lambda x_i}$ for ${i=1,\ldots,s}$. Then the number of non-proportional primitive solutions of~\eqref{euneq} is bounded by an effectively computable quantity depending only on~$r$ and~$s$. 
\end{theorem}
Note that~$\Gamma$ in \cite{AV09,ESS02} corresponds to our $\Gamma^s$, and~$r$ in \cite{AV09,ESS02} corresponds to our~$rs$.

\bigskip

In addition to this theorem, we will need a simple technical lemma.

{\sloppy

\begin{lemma}
\label{lprop}
Let ${\alpha,\beta,\gamma \in  \C^\times}$   and ${m,n,m',n'\in \Z}$ be such that 
$$
m\ne  n,  \quad m,n\ne 0,\quad
m'\ne n', \quad m',n'\ne 0, \quad(m,n)\ne (m',n'). 
$$
Consider the sets\footnote{Perhaps, it would be more proper to call them \textit{multi-sets}  because some of the listed numbers may accidentally be equal. However, we prefer to say simply ``sets'', hoping that this formal inaccuracy does not produce any confusion. }  
\begin{align*}
S&=\{\alpha^m\beta^n,\alpha^n\beta^m,\alpha^m\gamma^n,\alpha^n\gamma^m,\beta^m\gamma^n,\beta^n\gamma^m\},\\
S'&=\{\alpha^{m'}\beta^{n'},\alpha^{n'}\beta^{m'},\alpha^{m'}\gamma^{n'},\alpha^{n'}\gamma^{m'},\beta^{m'}\gamma^{n'},\beta^{n'}\gamma^{m'}\}.
\end{align*}
To every ${x\in S}$ we associate ${x'\in S'}$ in the obvious way (for instance, for ${x=\alpha^m\beta^n }$ we define ${x'=\alpha^{m'}\beta^{n'}}$). 
\begin{enumerate}
\item
\label{inotwotwotwo}
Assume that~$S$ admits a partition
${S=\{x_1,y_1\}\cup\{x_2,y_2\}\cup\{x_3,y_3\}}$   
into~$3$ two-element sets\footnote{To be precise, a partition of the multi-set~$S$ into~$3$ two-element muti-sets; say, we may have ${x_1=\alpha^m\beta^n}$, ${y_1=\alpha^n\gamma^m}$ or ${x_1=\alpha^m\beta^n}$, ${x_2=\alpha^n\gamma^m}$ even if accidentally ${\alpha^m\beta^n=\alpha^n\gamma^m}$.}  such that the  quotients ${x_1/y_1,x_2/y_2,x_3/y_3}$  are all roots of unity.  Then  one of the numbers ${\alpha/\beta, \alpha/\gamma,\beta/\gamma}$ is a root of unity.

\item
\label{ipartition}

Assume that~$S$ admits a partition ${S=T\cup U}$  into two sets (one of which is allowed to be empty) such that  for any ${x,y\in T}$ we have ${x/x'=y/y'}$, and the same holds true for any two elements of~$U$. Then  one of the numbers ${\alpha/\beta, \alpha/\gamma,\beta/\gamma}$ is a root of unity.

\end{enumerate}
\end{lemma}

}

\begin{proof}
In item~\ref{inotwotwotwo} we may assume, without loss of generality, that ${x_1=\alpha^m\beta^n}$. If ${y_1=\alpha^n\beta^m}$ then ${(\alpha/\beta)^{m-n}}$ is a root of unity, and we are done. Hence we may assume that ${x_2=\alpha^n\beta^m}$. If ${y_1=\alpha^m\gamma^n}$ then ${(\beta/\gamma)^n}$ is a root of unity, and  if ${y_1=\beta^n\gamma^m}$ then ${(\alpha/\gamma)^m}$ is a root of unity.  Hence we may assume that ${y_1\in \{\alpha^n\gamma^m,\beta^m\gamma^n\}}$, and, without loss of generality, ${y_1=\alpha^n\gamma^m}$. 

Similarly, we may assume that ${y_2\in \{\alpha^m\gamma^n,\beta^n\gamma^m\}}$. If ${y_2=\alpha^m\gamma^n}$ then ${\{x_3,y_3\}=\{\beta^m\gamma^n,\beta^n\gamma^m\}}$, and ${(\beta/\gamma)^{m-n}}$ is a root of unity. If ${y_2=\beta^n\gamma^m}$ then ${\{x_3,y_3\}=\{\alpha^m\gamma^n,\beta^m\gamma^n\}}$ and ${(\alpha/\beta)^m}$ is a root of unity.   This proves item~\ref{inotwotwotwo}.

In the proof of item~\ref{ipartition} we may assume that ${\#T\ge 3}$. 
Assume first that ${m=m'}$. Then ${n\ne n'}$.  We may assume that ${\alpha^m\beta^n\in T}$. Since ${\#T\ge 3}$, it must contain one of ${\alpha^n\beta^m,\alpha^m\gamma^n,\alpha^n\gamma^m,\beta^m\gamma^n}$. If, for instance, ${\alpha^n\beta^m\in T}$ then ${(\alpha/\beta)^{n-n'}=1}$. The other three cases are settled similarly. This completes the proof in the case ${m=m'}$. 

The case ${n=n'}$ is analogous. Now assume that ${m-m'=n-n'}$. Multiplying the elements of~$S$ and~$S'$ by  ${(\alpha\beta\gamma)^{-n}}$ and $ {(\alpha\beta\gamma)^{-n'}}$, respectively, we reduce this to the case ${m=m'}$, with ${m-n=m'-n'}$ as~$m$ and~$m'$,   with ${-n}$ as~$n$ and with ${-n'}$ as~$n'$. 

From now on
$$
m\ne m', \qquad n\ne n',\qquad m-m'\ne n-n'. 
$$
We may again assume that ${\alpha^m\beta^n\in T}$. If ${\alpha^m\gamma^n\in T}$ then ${(\beta/\gamma)^{n-n'}=1}$,  if ${\beta^n\gamma^m\in T}$ then ${(\alpha/\gamma)^{m-m'}=1}$, and if     ${\alpha^n\beta^m \in T}$ then ${(\alpha/\beta)^{(m-m')-(n-n')}=1}$. 

We are left with the case  
$$
T=\{\alpha^m\beta^n, \beta^m\gamma^n, \gamma^m\alpha^n\}, \qquad U=\{\alpha^n\beta^m, \beta^n\gamma^m, \gamma^n\alpha^m\}. 
$$
In this case we have
$$
\begin{array}{r@{\,}c@{\,}l}
\alpha^{m-m'}\beta^{n-n'}&= \beta^{m-m'}\gamma^{n-n'} &=\gamma^{m-m'}\alpha^{n-n'}, \\
\alpha^{n-n'}\beta^{m-m'}&= \beta^{n-n'}\gamma^{m-m'} &=\gamma^{n-n'}\alpha^{m-m'}.
\end{array} 
$$
Dividing term by term, we obtain
$$
(\alpha/\beta)^{(m-m')-(n-n')}=(\beta/\gamma)^{(m-m')-(n-n')}=(\gamma/\alpha)^{(m-m')-(n-n')}.  
$$
Since the product of the three numbers is~$1$, we obtain 
$$
(\alpha/\beta)^{3((m-m')-(n-n'))}=(\beta/\gamma)^{3((m-m')-(n-n'))}=(\gamma/\alpha)^{3((m-m')-(n-n'))}=1.  
$$
The lemma is proved. 
\end{proof}

Now we are ready to prove Theorem~\ref{thc}. As indicated in the Introduction, the argument is, essentially, due to Evertse et al \cite[Theorem~1]{EGST88}.  

Without loss of generality we may assume that ${\Omega=\{\alpha,\beta,\gamma\}}$ 
  such that neither of ${\alpha/\beta,\alpha/\gamma, \beta/\gamma }$ is a root of unity. Let~$\Gamma$ be the multiplicative group generated by~$\alpha$,~$\beta$,~$\gamma$ and $-1$. For a positive integer~$s$ let $\kappa(s)$ be the number of non-proportional primitive solutions of ${x_1+\cdots+x_s=0}$ in ${x_1, \ldots,x_s\in \Gamma}$. Clearly, 
$$
\kappa(1)=0, \qquad \kappa(2)=1,
$$
and Theorem~\ref{thuneq} implies that $\kappa(s)$ is bounded by an effectively computable quantity depending only on~$s$. We set ${\kappa(0)=1}$. 

Call a pair ${(m,n)\in \Z^2}$ with ${m>n>0}$ \textit{suitable} if there exists a trinomial of the form ${X^m+AX^n+B}$ vanishing at~$\Omega$. Note that the coefficients~$A$ and~$B$ are uniquely  determined in terms of $(m,n)$ using the formulas  
\begin{equation}
\label{eab}
A=-\frac{\alpha^m-\beta^m}{\alpha^n-\beta^n}, \qquad B=-\frac{\alpha^{m-n}-\beta^{m-n}}{\alpha^{-n}-\beta^{-n}}. 
\end{equation}
This means that the Theorem~\ref{thc} will be proved if we show that the number of suitable pairs $(m,n)$ is bounded by an absolute constant. 
We are going to show that this number does not exceed ${10\kappa(3)^2+ 15\kappa(4)+\kappa(6)}$.

Fix a suitable pair $(m,n)$. Since~$\alpha$,~$\beta$ and~$\gamma$ are roots of a trinomial of the form ${X^m+AX^n+B}$, we have 
\begin{equation}
\begin{vmatrix}
\alpha^m&\alpha^n&1\\
\beta^m&\beta^n&1\\
\gamma^m&\gamma^n&1
\end{vmatrix}=0. 
\end{equation}
This can be re-stated as follows: 
$$
(x_1, \ldots, x_6)=(\alpha^m\beta^n,-\alpha^n\beta^m,-\alpha^m\gamma^n,\alpha^n\gamma^m,\beta^m\gamma^n,-\beta^n\gamma^m) 
$$
is a solution of the equation ${x_1+\cdots+x_6=0}$.

This solution is not, in general,  primitive, but
item~\ref{inotwotwotwo} of Lemma~\ref{lprop} implies that there is a partition 
$$
\{1,2,\ldots,6\}=V\cup W
$$
such that the following holds:
\begin{itemize}
\item ${(\#V,\#W)\in \{(3,3),(4,2),(6,0)\}}$;
\item
${\sum_{i\in V}x_i=\sum_{i\in W}x_i=0}$; 
\item
no proper sub-sum of $\sum_{i\in V}x_i$ vanishes, and neither does any proper sub-sum of $\sum_{i\in W}x_i$; in other words, $(x_i)_{i\in V}$ and  $(x_i)_{i\in W}$ are primitive solutions of the corresponding equations. 
\end{itemize}
We will say that $(m,n)$ is a suitable pair of type $(V,W)$. 
(We identify types $(V,W)$ and $(W,V)$ when ${\#V=\#W=3}$.)

Now let $(m',n')$ be another suitable pair of the same type $(V,W)$. Item~\ref{ipartition} of Lemma~\ref{lprop} implies that either $(x_i)_{i\in V}$ and $(x_i')_{i\in V}$ are not proportional, or $(x_i)_{i\in W}$ and $(x_i')_{i\in W}$ are not. Hence there can exist at most ${\kappa(\#V)\kappa(\#W)}$ suitable pairs of a given type ${(V,W)}$.

Since a set of~$6$ elements admits~$10$ partitions of signature $(3,3)$ and~$15$ partitions of signature $(4,2)$, the total number of suitable pairs is bounded by 
$$
10\kappa(3)\kappa(3)+ 15\kappa(4)\kappa(2)+\kappa(6)= 10\kappa(3)^2+15\kappa(4)+\kappa(6). 
$$
The theorem is proved. \qed

\bigskip

Using the explicit bound from \cite[Theorem~6.2]{AV09}, one can produce a ridiculously big explicit value $10^{300000}$ for the constant in Theorem~\ref{thc}. Perhaps, this can improved by using some careful ad hoc arguments. 

\section{Proof of Theorem~\ref{thalg}}
\label{salg}

To start with, let us fix some conventions.

\begin{itemize}
\item
In this section we fix, once and for all, an embedding ${\bar\Q\hookrightarrow\C}$. 

\item

We say that $\log z$ is the \textit{principal value} of the complex logarithm of ${z\in \C}$ if ${-\pi <\Im\log z\le \pi}$. 

\end{itemize}

Our principal tool will be Baker's inequality in the form given by Matveev \cite[Corollary~2.3]{Ma00}, reproduced below.

\begin{theorem}[Matveev]
\label{thmat}
Let ${\theta_1, \ldots, \theta_s}$ be non-zero algebraic numbers belonging to a number field of degree~$d$, and  ${\log\theta_1, \ldots , \log\theta_s}$ some determinations of their complex logarithms.   Let ${b_1, \ldots, b_s\in \Z}$ be such that 
$$
\Lambda=b_1\log\theta_1+\cdots+b_s\log\theta_s\ne 0. 
$$
 Let ${A_1, \ldots, A_s, B}$ be real numbers satisfying 
\begin{align*}
A_k&\ge \max\{d\height(\theta_k),|\log\theta_k|,0.16\}\qquad (k=1,\ldots, s),\\
B&\ge \max\{|b_1|, \ldots, |b_s|\}. 
\end{align*}
Then 
\begin{equation*}
\log\Lambda\ge -2^{6s+20} d^2(1+\log d)A_1\cdots A_s (1+\log B). 
\end{equation*}
\end{theorem}

Here is a useful consequence.

\begin{corollary}
\label{cmat}
Let~$\alpha$ and~$\beta$ be non-zero algebraic numbers contained in a number field of degree~$d$, and~$k$ a positive integer. Assume that ${|\alpha|\ge |\beta|}$ and that ${\alpha^k\ne\beta^k}$. Then 
\begin{align}
\label{e>}
|\alpha^k-\beta^k|&\ge |\alpha|^k e^{-d^2(\height(\alpha/\beta)+1)}&& \text{if $|\alpha|>|\beta|$},\\
\label{e=}
|\alpha^k-\beta^k|&\ge |\alpha|^k e^{-10^{12}d^4(\height(\alpha/\beta)+1)\log (k+1)}&& \text{if $|\alpha|=|\beta|$}.
\end{align}
\end{corollary}

\begin{proof}
Set ${\theta=\beta/\alpha}$. 
If ${|\alpha|>|\beta|}$ then we apply Proposition~\ref{prmodnoone}:
$$
|\alpha^k-\beta^k|\ge |\alpha|^k(1-|\theta|^k)\ge |\alpha|^k(1-|\theta|)\ge |\alpha|^ke^{-d^2(\height(\theta)+\log 2)}. 
$$
Now assume that ${|\alpha|=|\beta|}$. In this case we use the equality
$$
|\alpha^k-\beta^k|= |\alpha|^k|1-\theta^k|
$$
We have ${|\theta|=1}$, and we let ${\log\theta}$ be the principal value of the logarithm; that is ${\log\theta=\lambda i}$ with ${-\pi<\lambda\le \pi}$. We set~$\ell$ to be the nearest integer to ${k\lambda/(2\pi)}$, and we define 
$$
\Lambda=k\log\theta-2\pi i \ell. 
$$
Note that ${|\ell|\le (k+1)/2}$, that $\Lambda$ is the principal value of $\log (\theta^k)$,  and that  ${\Lambda\ne 0}$ (because ${\alpha^k\ne \beta^k}$). 

If ${u\in [-\pi,\pi]}$ then ${|1-e^{ui}|\ge (2/\pi)|u|}$. Hence 
$$
|1-\theta^k|\ge (2/\pi)|\Lambda|.
$$
To estimate $|\Lambda|$ we use Theorem~\ref{thmat} with the following parameters: 
\begin{align*}
&s=2, \quad \theta_1=\theta, \quad \theta_2=-1, \quad b_1=k, \quad b_2=-2\ell, \\ 
&A_1=\pi d(\height(\theta)+1), \quad A_2=\pi, \quad B=k+1. 
\end{align*}
We obtain, after easy calculations, the estimate
$$
|\Lambda|\ge e^{-10^{11}d^4(\height(\theta)+1)\log (k+1)}.
$$
This proves~\eqref{e=}. 
\end{proof}

We will also need a simple, but crucial lemma.

\begin{lemma}
\label{labsval}
Let ${\alpha,\beta,\gamma\in \C^\times}$ be three roots of a trinomial with complex coefficients. Assume that ${|\alpha|=|\beta|=|\gamma|}$. Then one of the quotients ${\alpha/\beta,\alpha/\gamma,\beta/\gamma}$ is a root of unity.
\end{lemma}

\begin{proof}
Denote 
${R=|\alpha|=|\beta|=|\gamma|}$ and write our trinomial as ${X^m+AX^n+B}$ with ${B\ne 0}$. 
Consider the circles 
$$
\calC_1=\{R^me^{ui}:u\in [0,2\pi]\}, \qquad \calC_2=\{-B-AR^ne^{ui}:u\in [0,2\pi]\}.
$$
In other words,~$\calC_1$ is centered at~$0$ and has radius $R^m$ and~$\calC_2$  is centered at~$-B$ and has radius $|A|R^n$. The three numbers ${\alpha^m,\beta^m,\gamma^m}$ belong to~$\calC_1$, and the three numbers ${-B-A\alpha^n}$, etc. belong to~$\calC_2$. Since ${\alpha^m=-B-A\alpha^n}$, etc., the three numbers 
${\alpha^m,\beta^m,\gamma^m}$ belong to the intersection ${\calC_1\cap\calC_2}$, which may consist of 
two elements at most. Hence two of the numbers ${\alpha^m,\beta^m,\gamma^m}$ are equal, and this proves the lemma.
\end{proof}

Now we are ready to prove Theorem~\ref{thalg}. As we will see, a more natural parameter for our estimates is not $\height(\Omega)$, but the quantity 
$$
\tilheight(\Omega)=\{\max\height(\alpha/\beta): \alpha, \beta \in \Omega\}.
$$
Clearly, 
$
{\tilheight(\Omega) \le 2\height(\Omega)} 
$.

As in Section~\ref{sc} we may assume that ${\Omega=\{\alpha,\beta,\gamma\}}$, where none of the quotients ${\alpha/\beta,\alpha/\gamma,\beta/\gamma}$ is a root of unity. Lemma~\ref{labsval} implies that ${\alpha,\beta,\gamma}$ are not all three of the same absolute value, so we may assume that either ${|\alpha|>|\beta|\ge|\gamma|}$ or ${|\alpha|<|\beta|\le|\gamma|}$. If ${\alpha,\beta,\gamma}$ are roots of a trinomial, then ${\alpha^{-1},\beta^{-1},\gamma^{-1}}$ are roots of a trinomial  of the same degree and height. Hence we may assume that 
$$
|\alpha|>|\beta|\ge |\gamma|. 
$$

Let ${X^m+AX^n+B}$ be a trinomial vanishing at $\alpha,\beta,\gamma$. Recall that, as usual, ${m>n>0}$ and ${B\ne 0}$. We are going to estimate ${m-n}$ in terms of~$n$, and afterwards ${n}$ in terms of ${m-n}$. The two estimates will yield the desired conclusion.

\subsection{Estimating ${m-n}$ in terms of~$n$} 
We have
$$
A=-\frac{\alpha^m-\beta^m}{\alpha^n-\beta^n}=-\frac{\beta^m-\gamma^m}{\beta^n-\gamma^n},
$$
Using Corollary~\ref{cmat}, this implies the following lower and upper bounds for $|A|$:
\begin{align*}
|A|&\ge\frac{|\alpha^m-\beta^m|}{2|\alpha|^n}\ge\frac12|\alpha|^{m-n} e^{-d^2(\tilheight+1)},\\
|A|&\le\frac{2|\beta|^m}{|\beta^n-\gamma^n|}\le2|\beta|^{m-n} e^{10^{12}d^4(\tilheight+1)\log (n+1)}
\end{align*}
where we abbreviate ${\tilheight=\tilheight(\Omega)}$. This implies that 
\begin{equation*}
m-n\le \frac{10^{13}d^4(\tilheight+1)}{\log|\alpha/\beta|}\log(n+1). 
\end{equation*}
Proposition~\ref{prmodnoone} implies that  
\begin{equation}
\label{elowerlog}
\log |\alpha/\beta|\ge \frac12e^{-d^2(\tilheight+1)}, 
\end{equation}
and we obtain 
\begin{equation}
\label{emnn}
m-n\le 10^{16}e^{2d^2(\tilheight+1)}\log(n+1). 
\end{equation}
This is the promised estimate of ${m-n}$ in terms of~$n$.

\subsection{Estimating ${n}$ in terms of ${m-n}$} 

Write
$$
a=\alpha^{m-n}, \quad b=\beta^{m-n}, \quad c=\gamma^{m-n}. 
$$
Then 
$$
0=
\begin{vmatrix}
a\alpha^n&\alpha^n&1\\
b\beta^n&\beta^n&1\\
c\gamma^n&\gamma^n&1
\end{vmatrix}=
(a-b)(\alpha\beta)^n+ (b-c)(\beta\gamma)^n-(a-c)(\alpha\gamma)^n.
$$
Dividing by ${(a-c)(\alpha\gamma)^n}$, we obtain
$$
\bigl|\vartheta (\beta/\gamma)^n-1\bigr|= |\eta||\alpha/\beta|^{-n},
$$
where 
$$
\vartheta= \frac{a-b}{a-c}=\frac{1-(\beta/\alpha)^{m-n}}{1-(\gamma/\alpha)^{m-n}}, \qquad \eta=\frac{b-c}{a-c}=\frac{1-(\beta/\gamma)^{m-n}}{1-(\alpha/\gamma)^{m-n}}.
$$
We have clearly 
\begin{equation}
\label{ehtheta}
\height(\vartheta), \height(\eta) \le 2\tilheight(m-n)+2\log2. 
\end{equation}
In particular,
\begin{equation}
\label{eabseta}
|\eta| \le e^{2d(\tilheight(m-n)+1)}.
\end{equation}
If ${|\eta||\alpha/\beta|^{-n}\ge1/2}$ then
\begin{equation}
\label{estronger}
n \le \frac{\log(2|\eta|)}{\log|\alpha/\beta|} \le 10e^{2d^2(\tilheight+1)}(m-n),
\end{equation}
where we use~\eqref{elowerlog} to estimate $\log|\alpha/\beta|$ from below.

From now on we assume that 
\begin{equation}
\label{enearone}
\bigl|\vartheta (\beta/\gamma)^n-1\bigr|= |\eta||\alpha/\beta|^{-n} \le 1/2. 
\end{equation}
Let
$$
\log\vartheta=\log|\vartheta|+i\lambda, \qquad \log(\beta/\gamma)=\log|\beta/\gamma|+i\mu
$$
be the principal values of the complex logarithm,  that is, 
${-\pi <\lambda,\mu\le \pi}$. 
Let~$\ell$ be the nearest integer to ${(\lambda+n\mu)/(2\pi)}$. Clearly, ${|\ell|\le (n+2)/2}$.  Define
$$
\Lambda = \log\vartheta+n\log (\beta/\gamma)-2\pi i \ell. 
$$
Then~$\Lambda$ is the principal value of $\log(\vartheta (\beta/\gamma)^n)$, and we have ${\Lambda\ne 0}$, because ${\vartheta (\beta/\gamma)^n\ne 1}$. 

If~$z$ is a complex number satisfying ${|z-1|\le 1/2}$ and ${\log z}$ is the principal value of its logarithm then ${|\log z|\le 2|z-1|}$. Hence~\eqref{enearone} implies  the upper bound
\begin{equation}
\label{euplam}
|\Lambda|\le 2|\eta||\alpha/\beta|^{-n}. 
\end{equation}
To estimate $|\Lambda|$ from below, we use Theorem~\ref{thmat} with the following parameters:
\begin{align*}
&s=3, \quad \theta_1=\vartheta, \quad \theta_2=\beta/\gamma, \quad \theta_3=-1, \quad
b_1=1, \quad  b_2=n, \quad b_3=-2\ell, \\ 
&A_1=\pi d(m-n)(\tilheight+1), \quad A_2=\pi d(\tilheight+1), \quad A_3=\pi, \quad B=n+2. 
\end{align*}
(Note that ${\height(\vartheta)\le 2\tilheight(m-n)+2}$ by~\eqref{ehtheta}, which means that our choice of~$A_1$ fits the hypothesis of Theorem~\ref{thmat}.) 
We obtain
$$
|\Lambda|\ge e^{-10^{16}d^5(\tilheight+1)^2(m-n)\log(n+1)}
$$
Comparing this with the upper bound~\eqref{euplam} and taking into account the estimates~\eqref{elowerlog},~\eqref{eabseta}, we obtain
\begin{align}
n&\le \frac{10^{16}d^5(\tilheight+1)^2(m-n)\log(n+1)+\log(2|\eta|)}{\log|\alpha/\beta|}\nonumber\\
\label{eweaker}
&\le 10^{20}e^{2d^2(\tilheight+1)}(m-n)\log(n+1). 
\end{align}
Thus, we have either~\eqref{estronger} or~\eqref{eweaker}. Since the latter is formally weaker than the former, we always have~\eqref{eweaker}.

\subsection{Conclusion}

Substituting~\eqref{emnn} into~\eqref{eweaker}, we obtain 
$$
n\le 10^{36}e^{4d^2(\tilheight+1)}(\log(n+1))^2. 
$$
This implies that 
${n\le 10^{50}e^{5d^2(\tilheight+1)}}$. 
Substituting this into~\eqref{emnn}, we deduce that 
${m-n\le 10^{30}e^{3d^2(\tilheight+1)}}$. 
Hence 
${m \le 10^{60}e^{5d^2(\tilheight+1)}}$. 
Since ${\tilheight\le 2\height(\Omega)}$, this proves that 
$$
m\le 10^{60}e^{10d^2(\height(\Omega)+1)},
$$
which is the wanted bound for the degree of the trinomial. Finally, from~\eqref{eab} we deduce that 
$$
\height(A),\height(B) \le 10^{65}e^{10d^2(\height(\Omega)+1)},
$$
which implies the wanted bound for the height of the trinomial. Theorem~\ref{thalg} is proved.

{\footnotesize
\paragraph{Acknowledgments}
We thank Philipp Habegger for an encouraging discussion. We also thank the anonymous referee for pointing out several inaccuracies and for many other useful suggestions that helped us to improve the presentation. 

\bibliographystyle{amsplain}

\bibliography{bib}

}

\end{document}